\documentclass{article}
\usepackage{arxiv}
\usepackage[utf8]{inputenc} % allow utf-8 input
\usepackage[T1]{fontenc}    % use 8-bit T1 fonts
\usepackage{hyperref}       % hyperlinks
\usepackage{url}            % simple URL typesetting
\usepackage{booktabs}       % professional-quality tables
\usepackage{amsfonts}       % blackboard math symbols
\usepackage{nicefrac}       % compact symbols for 1/2, etc.
\usepackage{microtype}      % microtypography

\usepackage{amssymb,amsmath,amsthm}
\usepackage{enumitem}
\usepackage{hyperref}

\newtheorem{theorem}{Theorem}[section]
\newtheorem{lemma}{Lemma}[section]

\newtheorem{proposition}{Proposition}[section]
\newtheorem{definition}{Definition}[section]
\newtheorem{example}{Example}[section]
\def\Fn{\mathop{\rm Fn}\nolimits}
\def\CA{\mathop{\rm co}\nolimits}
\def\At{\mathop{\rm At}\nolimits}
\def\dom{\mathop{\rm dom}\nolimits}
\def\img{\mathop{\rm img}\nolimits}
\theoremstyle{nonumberplainnobrackets}

\title{An alternative to {\sl back-and-forth}}
\author{Tonatiuh Matos-Wiederhold}
\date{\vspace{-5ex}}

\begin{document}

\maketitle

\begin{abstract}
 We exhibit how the Rasiowa-Sikorski Lemma simplifies, in a sense, proofs of results that make use of the technique known as {\sl back-and-forth}, often resulting in not very illustrative arguments. The first two sections seek to show one simple and one complicated proofs of known results, in the hopes that the reader appreciates how the arguments end up, in our view, considerably clearer than those found in classic literature. The final section shows how the same techniques can be adapted to areas commonly considered distant to Set Theory, in this instance, Graph Theory.

 Sections one and two are based on my bachelor thesis (see \cite{tesis}), under the direction of Dr. Roberto Pichardo Mendoza, whom I deeply thank for his advice and revision of this work. All results mentioned in this paper are well known, however, as far as we know, the proofs in the last two sections are original.
\end{abstract}

\section{Introduction}
 All notation and basic results in Set Theory can be consulted in \cite{hernandez}. Among the essential notations we use, for example, for two arbitrary sets $A$ and $B$, $A\subseteq B$ means every element of $A$ is also an element of $B$. If $X$ is any set, $[X]^{<\omega}$ is the colection of all finite subsets of $X$. A set is countable if it has at most as many elements as the set of positive integers $\mathbb N$. As is common practice, $\omega$ denotes the first transfinite ordinal number, i.e., $\omega\setminus\{0\}=\mathbb N$.

 The symbols for {\sl domain} and {\sl image} of a function are also of use: if $f$ is a function, $\dom(f):=\{x\colon\exists y((x, y)\in f)\}$ and $\img(f):=\{y\colon\exists x((x, y)\in f)\}$. We also write $f[A]:=\{y\colon \exists a\in A((a,y)\in f)\}$ and $f^{-1}[B]:=\{x\colon \exists b\in B((x,b)\in f)\}$.
 
 Furthermore, $^B\!A$ is the collection of all functions from $B$ into $A$. So, for instance, ${}^{<\omega}2:=\bigcup_{n<\omega}{}^n2$ is the set of all finite binary sequences.
 
 Finally, given a function $f$ and a set $A$, we define the {\sl restriction} of $f$ to $A$ as $f\restriction A=\{(a,b)\in f\colon a\in A\}$.

\section{Preliminary}
 For the purposes of this work, a {\sl preorder} is an ordered pair $(\mathbb P,\leq)$ where $\leq$ is a reflexive and transitive relation over the set $\mathbb P$. If $\leq$ happens to also be antisymmetric, we say the ordered pair is a {\sl partial order}. We take for granted that every preorder (partial order) $\leq$ induces its corresponding strict preorder (partial order) (see \cite[Definición~4.101, p.~78]{hernandez}) and viceversa and denote it by $<$. From now on, fix $(\mathbb P,\leq)$ to be a preorder an let us define two important properties.
 
\begin{definition}\label{def:dense}
 A subset $D$ of $\mathbb P$ is called {\sl dense} if for all $p\in\mathbb P$ there exists $q\in D$ such that $q\leq p$.
\end{definition}

 We observe that in the previous definition, given the reflexivity of the order, it's equivalent substituting ``for all $p\in\mathbb{P}$'' by ``for all $p\in\mathbb{P}\setminus D$''.
 
\begin{definition}\label{def:filter}
 Let $p,q\in\mathbb{P}$ and $F\subseteq\mathbb{P}$.
\begin{list}{•}{}
 \item We say $p$ and $q$ are {\sl compatible} if there exists $r\in\mathbb{P}$ such that $r\leq p$ and $r\leq q$, and we abbreviate this fact as $p\mid q$. Otherwise, we say they are {\sl incompatible} and and use the symbol $p\perp q$.
 \item We call $F$ a {\sl filter} if:
 \begin{enumerate}
  \item for all $p,q\in F$ there exists $r\in F$ such that $r\leq p$ and $r\leq q$ and
  \item the conditions $p\in\mathbb{P}, q\in F$ and $q\leq p$ imply the set membership $p\in F$.
 \end{enumerate}
\end{list}
\end{definition}

 One simple observation is that item (2) of the definition of filter applies for any finite amount of elements in it. That is, if $F$ is a filter, $n$ a natural number and we have $\{p_i\colon i<n\}\subseteq F$, then there exists $r\in F$ such that for all $i<n$, $r\leq p_i$.

\begin{definition}\label{def:generic}
 If $F$ is a filter and $\mathcal{D}$ is a family of dense subsets of $\mathbb{P}$, we say $F$ is {\sl$\mathcal{D}$-generic} if $F$ intersects all the elements of $\mathcal{D}$, i.e. for all $D\in\mathcal{D}$, $F\cap D\neq\emptyset$.
\end{definition}

 Let's see an example that, in addition to illustrating the notions discussed so far, will prove useful in the following section.
  
\begin{example}
 Let $X$ and $Y$ be two nonempty sets. Denote by $\Fn(X,Y)$ the collection of all functions from some finite subset of $X$ into $Y$, in symbols, $$\Fn(X,Y):=\{f\in[X\times Y]^{<\omega}\colon f\text{ is a function}\}.$$ Then, directly from the definitions, the ordered pair $(\Fn(X,Y),\supseteq)$ is a partial order with maximum element $\emptyset$.
\end{example}

 Every function member of $\Fn(X,Y)$ is called a {\sl finite partial function of} $X$ into $Y$. Moreover, we prove our example has the following properties, pertinent to the aforementioned definitions .
 
\begin{theorem}\label{teo:Fn}~ 
\begin{enumerate}
 \item Given two functions in $\Fn(X,Y)$, they are compatible (according to definition~\ref{def:filter}) if and only if they are compatible as functions;
 \item if $F$ is a filter in $\Fn(X,Y)$, then $\bigcup F$ is a function;
 \item for all $x\in X$, the subset $D_x:=\{p\in\Fn(X,Y)\colon x\in\dom(f)\}$ is dense; and
 \item if $F$ is $\{D_x\colon x\in X\}$-generic, then $\dom(\bigcup F)=X$, i.e. $\bigcup F$ is a function from $X$ into $Y$.
\end{enumerate}
\end{theorem}

\begin{proof}~ 
\begin{enumerate}
 \item Let $p,q\in\Fn(X,Y)$. If $p\mid q$, then there exists $r\in\Fn(X,Y)$ with $r\supseteq p,q$, particularly, $r\supseteq p\cup q$, which implies $p\cup q\in\Fn(X,Y)$. On the other hand, if $p\cup q\in\Fn(X,Y)$, immediately we have $p\cup q\supseteq p,q$ and, as a consequence, $p\mid q$.
 \item From the previous point it follows that every filter $F$ forms a {\sl compatible system of functions}, i.e. $\bigcup F$ is a function.
 \item Fix $x\in X$, $D_x$ as in the statement and $p\in\Fn(X,Y)\setminus D_x$. Then, taking any $y_0\in Y$, $q:=p\cup\{(x,y_0)\}\in\Fn(X,Y)$ satisfies $q\supseteq p$.
 \item Under the assumption of the statement, is is clear that $\dom(\bigcup F)\subseteq X$. Given any $x\in X$, by the hypothesis of genericity, there is $p\in F\cap D_x$; i.e., $x\in\dom(p)\subseteq\dom(\bigcup F)$. This proves the other inclusion.
\end{enumerate}
\end{proof}

 One natural question is when we can guarantee, in general, the existence of a generic filter, like for instance the one described in the previous result. The following fact partially answers this question, which is sufficient for our purposes.
 
\begin{theorem}[Rasiowa-Sikorski Lemma]\label{teo:r-s}
 Given a preorder $(\mathbb P,\leq)$, $p\in\mathbb P$ and a \emph{countable} family of dense sets $\mathcal D$, there exists a $\mathcal D$-generic filter to which $p$ belongs.
\end{theorem}

\begin{proof}
 By the countability, we can write $\mathcal{D}=\{D_n\colon n\in\omega\}$ (possibly with repetitions). We will construct a sequence $\{q_n\}_{n<\omega}$ in $\mathbb{P}$ such that for all $n<\omega$:
 \begin{enumerate}[label=\alph*)]
 \item $q_{n+1}\leq q_n \leq q_0=p$ (the sequence is decreasing) and
 \item $q_{n+1} \in D_n$.
 \end{enumerate}
 
 Start by taking $q_0:=p$ and, if $q_n$ is already chosen, since $D_n$ is dense, there exists $q_{n+1}\in D_n$ such that $q_{n+1}\leq q_n$. This completes the recursion.
 
 We now verify that $F:=\{t\in\mathbb{P}\colon\exists n<\omega(q_n\leq t)\}$ is the sought filter. Note that by the reflexivity of $\leq$, we have $\{q_n\}_{n<\omega}\subseteq F$ and, in particular, $p\in F$. Regarding the properties of a filter:
 \begin{enumerate}
 \item Given $t_1, t_2\in F$, there exist $m,k<\omega$ such that $q_m\leq t_1$ and $q_k\leq t_2$. By (a) and by the transitivity of $\leq$, $q_{m+k}\leq t_1,t_2$ and $q_{m+k}\in F$.
 \item Given $q\in\mathbb{P}$ and $t\in F$, if $t\leq q$ then, by the transitive property of $\leq$, $q\in F$.
 \end{enumerate}
 
 Finally, by (b), $F$ is a $\mathcal{D}$-generic filter.
\end{proof}

\section{A characterization of the rationals}
 
 A first application of the Rasiowa-Sikorski Lemma is as follows. It is known that the rationals with their usual order $(\mathbb Q,\leq)$ are a dense, countable and unbounded total order. Cantor proved that these four properties characterize the rationals as an ordered set; i.e., up to isomorphism, there is a unique order with these four properties (see \cite[Teorema~11.2, p.~297]{hernandez}). Theorem~\ref{teo:r-s} yields quite a simple proof of the mentioned fact.
 
\begin{theorem}[Cantor]\label{teo:caractQ}
 Let $(X,\preceq)$ be a dense, countable and unbounded total order. Then $(X,\preceq)$ is isomorphic to $(\mathbb Q,\leq)$.
\end{theorem}

\begin{proof}
 We are going to construct the isomorphism using the theorem~\ref{teo:r-s}. For this, it is necessary to give a preorder $\mathbb P$ and a suitable countable family of dense sets.

 Let $\mathbb P:=\{p\in\Fn(X,\mathbb Q)\colon\forall x,y\in X(x\prec y\rightarrow p(x)<p(y))\}$, that is $\mathbb P $ is the collection of all finite partial functions that preserve the strict order. Again $\emptyset\in\mathbb P$. We sort $\mathbb P$ with the inverse set inclusion.
 
 We will show that, for each $x\in X$, the set $D_x:=\{p\in\mathbb P\colon x\in\dom(p)\}$ is dense in $\mathbb P$. With this in mind, pick $p\in\mathbb P\setminus D_x$ and define $$x^{\leftarrow}:=\{y\in\dom(p)\colon y\prec x\}\text{ and }x^{\rightarrow}:=\{y\in\dom(p)\colon x\prec y\}.$$ Then there are exactly three cases:
 \begin{enumerate}[label=\alph*)]
 \item $x^{\leftarrow}=\emptyset$. Since $\preceq$ is a total order, $\dom(p)=x^{\rightarrow}$. Given that $\mathbb Q$ is unbounded (in particular it has no minimum) and the set $\img(p)$ is finite, there exists $a_0\in\mathbb Q$ such that for all $a\in\img(p)$ we have $a_0<a$. Thus $q:=p\cup\{(x,a_0)\}$ is an element of $D_x$ that satisfies $q\supseteq p$.
 \item $x^{\rightarrow}=\emptyset$. A minimal modification to the above argument produces $q\in D_x$ such that $q\supseteq p$.
 \item $x^{\leftarrow}\neq\emptyset\text{ and }x^{\rightarrow}\neq\emptyset$. Let us denote by $a$ and $b$, respectively, the maximum element of $x^{\leftarrow}$ and the minimum element of $x^{\rightarrow}$ in $X$. Then $p(a)<p(b)$ and by density there exists a rational $c$ such that $p(a)<c<p(b)$. Therefore $q:=p\cup\{(x,c)\}$ is an element of $D_x$ with $q\supseteq p$.
 \end{enumerate}
 
 Similarly, for each $a\in\mathbb Q$ we define $E_a:=\{p\in\mathbb P\colon a\in\img(p)\}$ and, by an argument similar to the one set out above, it follows that $E_a$ is dense.
 
 We consider $\mathcal D:=\{D_x\colon x\in X\}\cup\{E_a\colon a\in\mathbb Q\}$ and note that since $X$ and $\mathbb Q$ are countable, $\mathcal D$ is a countable family of dense subsets of $\mathbb P$. So, by Theorem~\ref{teo:r-s}, there exists $F$, a $\mathcal D$-generic $\mathbb P$-filter with $\emptyset\in F$. As mentioned in Theorem~\ref{teo:Fn}(2), $f:=\bigcup F$ is a function. We claim $f$ is the sought isomorphism.

 Given $x\in X$, since $F$ is $\mathcal D$-generic, in particular there exists $p\in F\cap D_x$, that is, $x\in\dom(p)\subseteq\bigcup_{q\in F}\dom(q)=\dom(f)$. Thus, $\dom(f)=X$, that is, $f\colon X\longrightarrow\mathbb Q$. Similarly, for every $a\in\mathbb Q$ there is $p\in F\cap E_a$ and hence $a\in\img(p)\subseteq\img(f)$. Therefore $f$ is surjective.

 Suppose for $x,y\in X$ that $x\prec y$. Then there are $p,q\in F$ such that $x\in\dom(p)$ and $y\in\dom(q)$. But $F$ is a filter, so there exists $r\in F$ such that $r\supseteq p,q$. In particular $x,y\in\dom(r)$. Thus it is clear that $f(x)=r(x)<r(y)=f(y)$. In other words, $f$ preserves the strict order and in particular is injective. Therefore, $f$ is a bijection between two linearly ordered sets that preserves the strict order, that is, by virtue of \cite[Teorema~4.124,p.~84]{hernandez}, an order isomorphism.
\end{proof}

\section{There is a unique countable atomless Boolean algebra}

 In this section we extend the previous argument to a considerably more difficult result: we will prove that, up to isomorphism, there is a single countable atomless Boolean algebra (see Definition~\ref{def:atom}, below). We use the notation, basic results, and definitions (as well as the notion of {\sl Boolean algebra}) that appear in the third chapter of \cite{bool}.
 
 A proof with {\sl back-and-forth} is found in \cite[Teorema~5.4.9,pp.~132-133]{bool}.

  Given a Boolean algebra $(A,\leq)$ (to simplify the notation, we will refer to it as just $A$) and $a,b\in A$,
\begin{enumerate}
 \item the symbols $0_A$ and $1_A$ denote, respectively, the minimum and the maximum element of $A$, and we always assume that these are different;
 \item $a\wedge b$ and $a\vee b$ represent the infimum and the supremum, correspondingly;
 \item similarly, for $B\subseteq A$, $\bigwedge B$ and $\bigvee B$ denote the infimum and supremum of the set $B$, respectively;
 \item $a'$ represents the complement of the element $a$;
 \item $a-b$ is short for $a\wedge b'$; and
 \item we denote the set of positive elements of $A$ by $A^+$, that is, $A^+:=A\setminus\{0_A\}$.
\end{enumerate}

\begin{lemma}\label{lema:supremo}
 For any $a\in A^+$ and $E\in[A]^{<\omega}$, the condition $a\leq\bigvee E$ implies that there is $x\in E$ with $a\wedge x>0_A$.
\end{lemma}

\begin{proof}
 By counterpositive, suppose that $a\wedge x=0_A$ for each $x\in E$. Then, by distributivity, $$a=a\wedge\bigvee E=\bigvee\{a\wedge x\colon x\in E\}=0_A,$$ that is, $a\notin A^+$.
\end{proof}

\begin{definition}\label{def:atom}
 Let $A$ be a Boolean algebra. An element $a\in A^+$ is called an {\sl atom} if it is a $\leq$-minimal element in $A^+$. The set of atoms of $A$ is denoted by $\At(A)$. To simplify the notation, let us also define, for each $a\in A$, the set $$\At_a(A):=\{x\in\At(A)\colon x\leq a\}.$$ In this way, we say that $A$ is {\sl atomless} whenever $\At(A)=\emptyset$.
\end{definition}

 For a brief illustrative example, consider, for a nonempty set $X$, the Boolean algebra $(\mathcal P(X),\subseteq)$. In this case, given $Y\subseteq X$, $\At_Y(\mathcal P(X))=\{\{y\}\colon y\in Y\}$. Note that if $x,y\in X$ satisfy $\emptyset\neq\{x\}\cap\{y\}$, then $x=y$. We also have that $\bigcup\At_Y(\mathcal P(X))=Y$. This serves as motivation for the next two propositions.
 
 We have the following properties with respect to atoms in $A$. First, if two atoms of $A$, say $a$ and $b$, are such that $0_A<a\wedge b$, then, as $a\wedge b\leq a$ and because $a$ is a atom, $a\wedge b=a$, that is, $a\leq b$ and, since $b$ is also an atom, $a=b$. In short, we have proved the following.
 
\begin{proposition}\label{obs:atom}
 For any $a,b\in\At(A)$, if $0_A<a\wedge b$, then $a=b$.
\end{proposition}

\begin{proposition}\label{prop:atom}
 Suppose $ A $ is finite. Then,
\begin{enumerate}
 \item $\At(A)$ is dense in $A^+$ and
 \item for all $a\in A$, $a=\bigvee\At_a(A)$ (see Definition~\ref{def:atom}).
\end{enumerate}
\end{proposition}

\begin{proof}
 Let's show the first point by counterpositive. If it were not the case that $\At(A)$ is dense in $A^+$, then there would be an element $a\in A^+$ such that for all $b\in A$, $b\leq a$ implies $b\notin\At(A)$. We will argue that this last condition results in the existence of a strictly decreasing sequence $\{a_k\colon k<\omega\}\subseteq A$, especially with an infinity of elements, which guarantees that $A$ is infinite.
 
 Since in particular $a_0:=a$ is not an atom, there exists $a_1\in A^+$ such that $a_1<a$. If, for some natural number $n$, we have constructed $\{a_k\colon k\leq n\}$, a strictly decreasing sequence, then, by hypothesis, $a_n$ is not an atom (since $a_n\leq a$), and thus we obtain $a_{n+1}$ with the desired properties. This completes our recursion.
 
 Now thinking about the second item, it is clear that $a$ is the upper bound of $\At_a(A)$ and thus $b:=\bigvee\At_a(A)\leq a$. On the other hand, if it happened that $a-b\in A^+$, then the first item of the lemma would throw us an atom $c\in A^+$ such that $c\leq a-b$; then $c\leq a$ and $c\leq b'$. The first of these inequalities, together with $c\in\At(A)$, tells us that $c\leq b$, which clearly, being $c$ positive, contradicts the second inequality. Therefore, $a-b=0_A$, that is, $a=b$.
\end{proof}

 From the first item of the previous proposition it follows that the equality $\At(A)=\emptyset$ implies that $A$ is infinite (note that $\emptyset$ cannot be dense in $A$ since $1_A\in A^+$). Moreover if $A$ is countable, then $|A|=\omega$. This is important for the central result of the section, since all countable and atomless algebra have cardinality $\omega$.

  Suppose that $A_0$ is a subalgebra of $A$ and let $u\in A$. Since the intersection of subalgebras is once again a subalgebra, for any arbitrary subset of $A$ it makes sense to define the smallest subalgebra that contains it. In particular, we have the following.

\begin{definition}\label{def:simple}
 The {\sl simple extension of $A_0$ given by $u$} is the smallest subalgebra of $A$ that contains $A_0\cup\{u\}$ and is denoted by $A_0(u)$.
\end{definition}

 For the following lemmas, we will assume that $C_0$ and $D_0$ are Boolean algebras and that $C$ and $D$ are subalgebras of $C_0$ and $D_0$ respectively. We will stop specifying which algebra the maximum and minimum and their respective orders belong to as it is clear from the context and simplifies the reading of the proofs.

\begin{lemma}\label{lema:simple}
 Let $u$ be an arbitrary element of $C_0$. We have that $$C(u)=\{(a\wedge u)\vee(b-u)\colon a,b\in C\}.$$ If, in addition, $H\colon C(u)\longrightarrow D_0$ is a Boolean homomorphism, then $H[C(u)]$ is the simple extension of the subalgebra $H[C]$ given by $H(u)$, that is, $H[C(u)]=H[C](H(u))$.
\end{lemma}

\begin{proof}
 In the proof that follows we freely use the basic properties of Boolean operators, such as associativity, commutativity, distributivity, De Morgan's Laws, and equivalences with the order.

 First, let's see that $C':=\{(a\wedge u)\vee(bu)\colon a,b\in C\}$ is a subalgebra of $C_0$ containing $C\cup\{u\}$, and so then by minimality, $C(u)\subseteq C'$.
 
 First of all, since $C$ is a subalgebra of $C_0$, $\{0,1\}\subseteq C$, and so $$u=(1\wedge u)\vee(0-u)\in C'.$$
 
 Also, given $\{a,a_1,b,b_1\}\subseteq C$, we have $$a=a\wedge(u\vee u')=(a\wedge u)\vee(a-u)\in C',$$ that is, $C\cup\{u\}\subseteq C'$. Then we can use that (again for being a subalgebra) both $a\vee a_1$ and $b\vee b_1$ are elements of $C$ to get
\begin{align*}
 [(a\wedge u)\vee(b-u)]\vee[(a_1\wedge u)\vee(b_1-u)]&=\\
 [(a\wedge u)\vee(a_1\wedge u)]\vee[(b-u)\vee(b_1-u)]&=\\
 [(a\vee a_1)\wedge u]\vee[(b\vee b_1)-u]&\in C'.
\end{align*}

 To see that $C'$ is closed under complements, first notice that
\begin{align*}
 (a'\wedge b')\wedge[(b'\wedge u')\vee(a'\wedge u)]'&=\\
 (a'\wedge b')\wedge[(b'\wedge u')'\wedge(a'\wedge u)']&=\\
 b'\wedge(b'\wedge u')'\wedge a'\wedge(a'\wedge u)'&=\\
 b'\wedge(b\vee u)\wedge a'\wedge(a\vee u')&=\\
 b'\wedge u\wedge a'\wedge u'&=0.
\end{align*}
 And as a consequence,
\begin{align}\label{lema:simple_obs}
 a'\wedge b'\leq(b'\wedge u')\vee(a'\wedge u).
\end{align}

 Finally, since $\{a',b'\}\subseteq C$,
\begin{align}\label{lema:simple_2}
\begin{split}
 [(a\wedge u)\vee(b-u)]'&=\\
 (a\wedge u)'\wedge(b-u)'&=\\
 (a'\vee u')\wedge(b'\vee u)&=\\
 [(a'\wedge b')\vee(a'\wedge u)]\vee[(u'\wedge b')\vee(u\wedge u')]&=\\
 (a'\wedge b')\vee(a'\wedge u)\vee(u'\wedge b')&=\quad\text{(by (\ref{lema:simple_obs}))}\\
 (a'\wedge u)\vee(b'-u)&\in C'.
\end{split}
\end{align}

 In this way, $C'$ is a subalgebra of $C_0$ containing $C\cup\{u\}$ and with that we have the first inclusion: $C(u)\subseteq C'$. The second inclusion is simpler, since given $a,b\in C$, that $C(u)$ is a subalgebra that contains $C\cup\{u\}$ implies that $(a\wedge u)\vee(b-u)\in C(u)$. We conclude that $C(u)=C'$ as desired.
 
 Let us assume now that $H\colon C(u)\longrightarrow D_0$ is a Boolean homomorphism. From what we have just shown in the previous paragraphs, just see that $$H[C(u)]=\{(a\wedge H(u))\vee(b-H(u))\colon a,b\in H[C]\}.$$
 
 On the one hand, given $a,b\in H[C]$, there are $x,y\in C$ with $H(x)=a$ and $H(y)=b$. Since $H$ is a homomorphism,
\begin{align*}
 (a\wedge H(u))\vee(b-H(u))=(H(x)\wedge H(u))\vee(H(y)-H(u))\\
 =H((x\wedge u)\vee(y-u))\in H[C(u)]
\end{align*}
 On the other hand, if $a\in H[C(u)]$, then there exist (using the first item of the lemma) $x,y\in C$ such that $a=H((x\wedge u)\vee(y-u))$. Since $H$ is a homomorphism, a similar calculation as above reveals that $a=(H(x)\wedge H(u))\vee(H(y)-H(u))$. Therefore, $H[C(u)]$ is the simple extension of $H[C]$ given by $H(u)$.
\end{proof}

 The following extension lemma will be very useful to us.

\begin{lemma}\label{lema:ext}
 Suppose that $u\in C_0$ and $w\in D_0$ are arbitrary. If a Boolean isomorphism $h\colon C\longrightarrow D$ satisfies that for all $x\in C$
\begin{enumerate}
 \item $x\leq u'$ is equivalent to $h(x)\leq w'$ and
 \item $x\leq u$ if and only if $h(x)\leq w$,
\end{enumerate}
 then there exists a Boolean isomorphism $H\colon C(u)\longrightarrow D(w)$ such that $h\subseteq H$ and $H(u)=w$.
\end{lemma}

\begin{proof}
 Let us assume the hypotheses of the lemma and let $x\in C(u)$. We want to define $H(x)$. By the previous lemma, $x$ has a representation in terms of two elements of $C$ and $u$. Although this representation for $x$ is not unique, we will see that the hypotheses that $h$ satisfies have the consequence that, for any $\{a_0,a_1,b_0,b_1\}\subseteq C$, the equalities
\begin{align}\label{lema:ext_1}
 x=(a_0\wedge u)\vee(b_0-u)&=(a_1\wedge u)\vee(b_1-u)
\end{align}
 imply that
\begin{align}\label{lema:ext_2}
 (h(a_0)\wedge w)\vee(h(b_0)-w)&=(h(a_1)\wedge w)\vee(h(b_1)-w).
\end{align}
  
 With that in mind, let us do the following. First, the condition (\ref{lema:ext_1}) implies that\vspace*{-\baselineskip}
\begin{align}\label{obs:function}
\begin{split}
 \\&a_0\wedge u=x\wedge u=a_1\wedge u\\
 \text{and}\quad&b_0-u=x-u=b_1-u;\\
 \text{then, for all}\ i<2,\qquad&(a_i-a_{1-i})\wedge u=0_A\\
 \text{and}\quad&(b_i-b_{1-i})-u=0_A;\\
 \text{from where}\qquad&(a_i-a_{1-i})\leq u'\\
 \text{and}\quad&(b_i-b_{1-i})\leq u.
\end{split}
\end{align}

 Now, using the two hypotheses of the lemma and that $h$ is a Boolean homomorphism, we obtain

\begin{align}\label{obs:function2}
\begin{split}
 \text{for}\ i<2,\quad&h(a_i)-h(a_{1-i})\leq w'\\
 \text{and}\quad&h(b_i)-h(b_{1-i})\leq w
\end{split}
\end{align}

 The latter inequalities imply that $$h(a_0)\wedge w=h(a_1)\wedge w\quad\text{and}\quad h(b_0)-w=h(b_1) -w$$ and, with that, condition (\ref{lema:ext_2}).
 
 Let's define then, for $x=(a\wedge u)\vee(b-u)$, $H(x):=(h(a)\wedge w)\vee(h(b)-w)$. $H$ is well defined by everything done previously. Especially if $a\in C$, we can write $a=(a\wedge u)\vee(a-u)$ to calculate $H(a)=(h(a)\wedge w)\vee(h(a)-w)=h(a)$, that is, $H\restriction C=h$. We can also write $u=(1\wedge u)\vee(0-u)$ to see that $H(u)=(h(1)\wedge w)\vee(h(0)-w)=(1\wedge w)\vee(0-w)=w$.

 It only remains to verify that, indeed, $H$ is an isomorphism. Let's arbitrarily take two elements $x,y$ in $C(u)$ and, by the previous lemma, suppose they look like $x=(a\wedge u)\vee(b-u)$ and $y=(a_1\wedge u)\vee(b_1-u)$, for some $\{a,a_1,b,b_1\}\subseteq C$. As we did in the proof of Lemma~\ref{lema:simple}, $x\vee y=[(a\vee a_1)\wedge u]\vee[(b\vee b_1)-u]$ and , as a consecuense,
\begin{align*}
 H(x\vee y)&=\\
 H([(a\vee a_1)\wedge u]\vee[(b\vee b_1)-u])&=\\
 (h(a\vee a_1)\wedge w)\vee(h(b\vee b_1)-w)&=\\
 (h(a)\vee h(a_1))\wedge w)\vee(h(b)\vee h(b_1))-w)&=\\
 ((h(a)\wedge w)\vee(h(a_1)\wedge w))\vee((h(b)-w)\vee(h(b_1))-w))&=\\
 ((h(a)\wedge w)\vee(h(b)-w))\vee((h(a_1)\wedge w)\vee(h(b_1)-w))&=\\
 H(((a\wedge u)\vee(b-u)))\vee H(((a_1\wedge u)\vee(b_1-u)))&=H(x)\vee H(y).
\end{align*}
 
 By an argument similar to the one we use to deduce (\ref{lema:simple_obs}) (writing $h(a)$, $h(b)$ and $w$ instead of $a$, $b$ and $u$, respectively), we get that
\begin{align}\label{obs:quick}
 h(a)'\wedge h(b)'\leq(h(b)'-w)\vee(h(a)'\wedge w).
\end{align} 
 Using this and the calculation that was done in (\ref{lema:simple_2}), we have that
\begin{align*}
 H(x)'=
 H([(a\wedge u)\vee(b-u)])'&=\\
 [(h(a)\wedge w)\vee(h(b)-w)]'&=\\
 (h(a)'\vee w')\wedge(h(b)'\vee w)&=\\
 (h(a)'\wedge h(b)')\vee(h(a)'\wedge w)\vee(w'\wedge h(b)')\vee(w'\wedge w)&=\\
 (h(a)'\wedge h(b)')\vee(h(a)'\wedge w)\vee(w'\wedge h(b)')&=\quad\text{(by (\ref{obs:quick}))}\\
 (h(a)'\wedge w)\vee(h(b)'-w)&=\\
 (h(a')\wedge w)\vee(h(b')-w)&=\\
 H((a'\wedge u)\vee(b'-u))&=\\
 H([(a\wedge u)\vee(b-u)]')&=H(x').
\end{align*}

 We thus conclude that $H$ is a Boolean homomorphism. Furthermore, Lemma~\ref{lema:simple} implies that $H[C(u)]=H[C](H(u))$ and since $H(u)=w$ and $h$ is an isomorphism extended by $H$, it follows that $H[C](H(u))=h[C](w)=D(w)$. We conclude that $ H\colon C(u)\longrightarrow D(w)$ is a Boolean epimorphism. It is worth noting that, from our lemma hypotheses, so far we haven't used the return implications of items (1) and (2).
 
 Let us see that $H$ is injective. Given $x\in C(u)$ such that $H(x)=0$, set $a,b\in C$ with $x=(a\wedge u)\vee(b-u)$. So we have the equality $(h(a)\wedge w)\vee(h(b)-w)=0$, which, in turn, has as a consequence that $(h(a)\wedge w)=(h(b)-w)=0$ or, seen in another way, $h(a)\leq w'$ and $h(b)\leq w$. Applying our two hypotheses of the lemma, we obtain that $a\leq u'$ and $b\leq u$, then $x=(a\wedge u)\vee(b-u)=0\vee0=0$. We checked that $H$ has a trivial kernel, and this implies that $H$ is injective. With this we have that $H$ is an isomorphism.
\end{proof}

 Intuitively, the lemma we just proved tells us that if $w$ is compared to $D$ in the same way that $u$ is compared to $C$, then we can extend any isomorphism between $C$ and $D$ to include $u$ and $w$. However, for the test of the central result of this section, it is very useful to know how to find an element $w\in D_0$ that does what is required by the lemma. The following result has exactly this objective.
 
\begin{lemma}\label{lema:hayw}
 Now suppose that $C$ and $D$ are finite, $u\in C_0$, $D_0$ is atomless, and that $h\colon C\longrightarrow D$ is a Boolean isomorphism. Then there exists $w\in D_0$ that satisfies all the hypotheses of Lemma~\ref{lema:ext}.
\end{lemma}

\begin{proof}
 Start by defining the following sets of atoms.
\begin{align*}
 A^0&:=\{a\in\At(C)\colon a\leq u\},\\
 A^1&:=\{a\in\At(C)\colon a\leq u'\}\text{ and}\\
 A^{1/\pi}&:=\At(C)\setminus(A^0\cup A^1).
\end{align*}
 It is clear that these three are disjoint and their union is $\At(C)$.
 
 Since $\At(D_0)=\emptyset$, for every $y\in A^{1/\pi}$, there exists $\hat y\in D_0^+$ such that $\hat y<h(y)$. With all this, we propose $$w:=\bigvee\{\hat y\colon y\in A^{1/\pi}\}\vee\bigvee h[A^0].$$ Let's see that $w$ fulfills the biconditionals that Lemma~\ref{lema:ext} has as hypotheses.
 
 We first prove that items (1) and (2) of Lemma~\ref{lema:ext} are true when $x$ is an atom, and then we see that the result extends to an arbitrary $x$. With that in mind, let's take an atom $x\in C$.
 
 Suppose $x\leq u'$; in particular, $x\in A^1$. Thus, by distributivity, $$h(x)\wedge w=\bigvee\{h(x)\wedge\hat y\colon y\in A^{1/\pi}\}\vee\bigvee\{h(x)\wedge h(y)\colon y\in A^0\}.$$ For each $y\in A^{1/\pi}$, $h(x)\wedge\hat y\leq h(x)\wedge h(y)$. Since $h$ is an isomorphism, it must send different atoms into different atoms, thus, by Proposition~\ref{obs:atom}, $h(x)\wedge h(y)=0$. Similarly, given any $y\in A^0$, $h(x)\wedge h(y)=0$. This proves that $h(x)\wedge w=0$, and then $h(x)\leq w'$.
 
 Now let us assume that $x\leq u$. By our definition, this means that $x\in A^0$, which implies that $h(x)\leq\bigvee h[A^0]\leq w$.
 
 Now consider the case where $h(x)\leq w'$. Then $h(x)\wedge w=0$, from where, by the choice of $w$ and by distributivity, we obtain that for any $y\in A^0$, $h(x)\wedge h(y)=0$ and, for any $y\in A^{1/\pi}$, $h(x)\wedge\hat y=0$. But $x$ is an atom and trivially $0<h(x)=h(x)\wedge h(x)$, so $x\notin A^0$; also, if $x$ were an element of $A^{1/\pi}$, then $0<\hat x=h(x)\wedge\hat x$, which contradicts what was said at the beginning of this paragraph , that is $x\notin A^{1/\pi}$. In short, $x\in A^1$, and then $x\leq u'$.
 
 Finally, if $h(x)\leq w$, then by Lemma~\ref{lema:supremo}, one of two cases must happen. The first case is that $0<h(x)\wedge\bigvee h[A^0]$. So, by the same lemma, there must be an atom $y\in A^0$ such that $0<h(x)\wedge h(y)$. Since $h$ is an isomorphism, it follows that $h(x)$, $h(y)\in\At(D)$ and, according to Proposition~\ref{obs:atom}, we conclude that $h(x)=h(y)$. Thus, by injectivity, $x=y\in A^0$. That is, $x\leq u$.
 
 The second case, $0<h(x)\wedge\bigvee h[A^0]$, implies that, by Lemma~\ref{lema:supremo}, there is some $y\in A^{1/\pi}$ for which $0<h(x)\wedge\hat y\leq h(x)\wedge h(y)$. Then, by Proposition~\ref{obs:atom}, $h(x)=h(y)$ and thus $x=y\in A^{1/\pi}$. In particular, we have that for all $z\in A^0$, $h(x)\wedge h(z)=0$ (since $z$ and $y$ belong to different classes of atoms), which has as a consequence that $h(x)\wedge\bigvee h[A^0]=0$. Consequently, only the case analyzed in the previous paragraph occurs.
 
 Now let us see that for an arbitrary $x\in C$ the item (1) and (2) of Lemma~\ref{lema:ext} are valid. Proposition~\ref{prop:atom} allows us to write $x=\bigvee\At_x(C)$, which, together with the definition of supremum, tells us that for any $v\in C_0$,
\begin{align}\label{obs:equi}
 x\leq v\qquad\text{if and only if}\qquad\text{for all }a\in\At_x(C),\ a\leq v.
\end{align} 
 We then have that (taking $v=u$ in (\ref{obs:equi})) $x\leq u$ if and only if, for all $a\in\At_x(C)$, $a\leq u$, which, from the previous paragraphs, is equivalent to the fact that for all $a\in\At_x(C)$, $h(a)\leq w$. The latter implies that $$h(x)=h\left(\bigvee\At_x(C)\right)=\bigvee\{h(a)\colon a\in\At_x(C)\}\leq w.$$ Conversely, if $h(x)\leq w$, then since $h$ is an isomorphism, for all $a\in\At_x(C)$, $h(a)\leq h(x)\leq w$. Since $a$ is an atom, it follows that $a\leq u$ and by (\ref{obs:equi}), $x\leq u$. This proves point (1).
 
 Similarly, substituting $u$ and $w$ for $u'$ and $w'$, respectively, in the previous paragraph, item (2) is obtained.
\end{proof}

\begin{theorem}\label{teo:iso_bool}
 If $(\mathbb A,\leq)$ and $(\mathbb B,\preceq)$ are (non-trivial) countable atomless Boolean algebras, then they are isomorphic.
\end{theorem}

\begin{proof}
 We construct the isomorphism using Rasiowa-Sikorski. With this in mind, define the set $\mathbb P$ using the following formula.
\begin{align*}
 p\in\mathbb P\text{ if and only if }p\in\Fn(A,B),\ p\text{ is a Boolean isomorphism and }\\\dom(p)\text{ and }\img(p)\text{ are finite subalgebras of }A\text{ and }B\text{, respectively}.
\end{align*}

 We give $\mathbb P$ the order $q\leq p$ if and only if $q$ is an extension of $p$, in short, $p\subseteq q$. We have that $(\mathbb P,\leq)$ is a partial order with maximum element $\{(0_A,0_B),(1_A,1_B)\}$. For ease of notation, let's abbreviate, for $p\in\mathbb P$, $A_p:=\dom(p)$ and, similarly, $B_p:=\img(p)$.
 
 For each element $a\in A$, we claim that $D_a:=\{p\in\mathbb P\colon a\in A_p\}$ is a dense subset of $\mathbb P$. For this, let's arbitrarily take $p\in\mathbb P\setminus D_a$, that is, $a\in A\setminus A_p$. Then apply lemmas \ref{lema:hayw} and \ref{lema:ext} (in that order) taking $C_0:=A$, $D_0:=B$, $C:=A_p$, $D:=B_p$, $h:=p$ and $u:=a$ to obtain, respectively, $w\in B$ and then an isomorphism $H\colon A_p(a)\longrightarrow B_p(w)$ that extends $p$. As the simple extension of a finite subalgebra becomes finite again and $a\in A_p(a)$, it turns out that $H\in D_a$, which leads to the conclusion that $D_a$ is dense.
 
 Similarly, for each $b\in B$, the set $E_b:=\{p\in\mathbb P\colon b\in B_p\}$ is dense in $\mathbb P$: given $p\in\mathbb P\setminus E_b$, we apply lemmas \ref{lema:hayw} and \ref{lema:ext} taking $C_0:=B$, $D_0:=A$, $C:=B_p$, $D:=A_p$, $h:=p^{- 1}$ and $u:=b$ to obtain, respectively, $w\in A$ and then an isomorphism $H\colon B_p(b)\longrightarrow A_p(w)$ which extends $p^{-1}$. Thus, the isomorphism $H^{- 1}\colon A_p(w)\longrightarrow B_p(b)$ is an extension of $p$ that has the element $b$ in its image so, analogous to the previous case, culminates in the proof that $E_b$ is dense in $\mathbb P$.
 
 Thus, the family $$\mathcal D:=\{D_a\colon a\in A\}\cup\{E_b\colon b\in B\}$$ is a countable collection (since $A$ and $B$ are countable) of dense sets in $\mathbb P$. By Theorem~\ref{teo:r-s}, there is a filter $F\subseteq\mathbb P$ which is $\mathcal{D}$-generic. Define $f:=\bigcup F$ and note that Theorem~\ref{teo:Fn}(2) guarantees that $f$ is a function. Let us verify that it is the sought isomorphism.
 
 First, note that given $a\in A$, by genericity, there exists $p\in G$ such that $a\in A_p\subseteq\dom(f)$. Similarly, for each $b\in B$, $b\in\img(f)$. Thus $ f\colon A\longrightarrow B$ is surjective.
 
 Given $a_0,a_1\in A$, there must be $p_0,p_1\in F$ such that for $i<2$, $a_i\in\dom(p_i)$; but $F$ is a filter, so there must be a condition $q\in F$ that extends both $p_0$ and $p_1$; especially $\{a_0,a_1\}\subseteq\dom(q)$. But $q$ is an isomorphism, so $a_0\leq a_1$ if and only if $f(a_0)=q(a_0)\preceq q(a_1)=f(a_1)$. Therefore, $f$ is an order isomorphism and therefore a Boolean isomorphism.
\end{proof}

 The rest of the section is dedicated to proving that the previous theorem is not vacuously true, that is, we concentrate on producing an example of a countable atomless Boolean algebra.
 
  Much of what is discussed below can be done in more general topological spaces, however, for the purposes of this work, it will suffice to particularize several concepts. Those that are not explicitly defined here should be understood as they appear in \cite{casa}.
 
 Let's consider the topological space ${}^\omega2$ obtained as the Tychonoff product of $\omega$ copies of the discrete space $2$. If $\CA({}^\omega2)$ denotes the collection of sets that are open and closed in ${}^\omega2$, it can be verified (see \cite[Ejemplo~3.1.6,p.~64]{bool}) that $(\CA({}^\omega2),\subseteq)$ is a Boolean algebra with $\emptyset$ and ${}^\omega2$ as minimum and maximum, respectively. Also, given $A,B\in\CA({}^\omega2)$, $A\wedge B=A\cap B$ and $A\vee B=A\cup B$.
 
 We are going to use the fact that, since every factor of the space is finite and therefore compact, by Tychonoff's Theorem, ${}^\omega2$ is compact as well.
 
 Recall that the product topology is the {\sl weak topology induced} by the family of projections (for details see section 4.3 of \cite{casa}), that is, in our particular case, the topology of ${}^\omega2$, denoted by $\tau({}^\omega2)$, is the smallest topology that makes every function $\pi_n\colon{}^\omega2\longrightarrow 2$ given by $\pi_n(f):=f(n)$ continuous. Since $2$ is a discrete space, a set $U\subseteq{}^\omega2$ is a canonical basic open set of ${}^\omega2$ if and only if there are $F\in[\omega]^{<\omega}\setminus\{\emptyset\}$ and $\{A_n\colon n\in F\}\subseteq\mathcal P(2)$ such that $$U=\bigcap_ {n\in F}\pi_n^{-1}[A_n].$$ It is worth noting that, due to the continuity of each projection, $U$, in addition to being a canonical open set, is an intersection of closed sets and, therefore, closed.
 
 Continuing with our choice of $F$ and $U$, suppose now that $U$ is nonempty, that is, $U\in\CA({}^\omega2)^+$. If we take, because $F$ is finite, $m\in\omega\setminus(\bigcup F+1)$ and any point $x\in U$, it is clear, according to the discussion in the previous paragraph, that $$x\in U':=\bigcap_{n<m}\pi_n^{-1}[\{x(n)\}]\in\tau({}^\omega2).$$ Also, given any $z\in U'$, the choice of $U'$ tells us that for all $n<m$, $z(n)=x(n)$, in particular, for all $n\in F\subseteq m$, $z(n)\in\{x(n)\}$, that is, $z\in U$. We have shown that $x\in U'\subseteq U$.
 
 Let us define now, for each $s\in{}^{<\omega}2$, $[s]:=\{x\in{}^\omega2\colon s\subseteq x\}$ and note that, if $x$ is any point of the product, $$z\in\bigcap_{n <m}\pi_n^{-1}[\{x(n)\}]$$ if and only if for all $n<m$, $z(n)=x(n)$ or, equivalently, $z\restriction m=x\restriction m$. More succinctly, $$\bigcap_{n<m}\pi_n^{-1}[\{x(n)\}]=[x\restriction m].$$ Thus, each set of the form $[x\restriction m]$ is an element of $\CA({}^\omega2)$.
 
\begin{proposition}\label{prop:base}
 The countable collection $\mathcal B:=\{[s]\colon s\in{}^{<\omega}2\}\subseteq\CA({}^\omega2)$ is a basis for the space $({}^\omega2,\tau({}^\omega2))$.
\end{proposition}

\begin{proof}
 Since $|{}^{<\omega}2|=\left|\bigcup_{n<\omega}{}^n2\right|\leq\omega\cdot\omega=\omega$, $\mathcal B$ is countable. Furthermore, the final comment in the paragraph preceding this statement results in $\mathcal B\subseteq\CA({}^\omega2)$.
 
 Given a canonical basic open set $U$ in ${}^\omega2$ and $x\in U$, the previous paragraphs to this proposition justify the existence of a natural number $m$ such that $$x\in[x\restriction m]\subseteq U.$$ Consequently, $\mathcal B$ is a basis of the space in question.
\end{proof}

 Let us now show that the Boolean algebra $\CA({}^\omega2)$ is atomless: given $A\in\CA({}^\omega2)^+$, we can choose $x\in A$ and, by the previous proposition, some $m<\omega$ such that $x\in[x\restriction m]\subseteq A$. Note the inclusion $[x\restriction(m+1)]\subseteq[x\restriction m]$; moreover, the function $f\in{}^\omega2$ given by $$f(n):=\begin{cases}x(n),\quad&n\neq m\\1-x(n),\quad&n=m\end{cases},$$ satisfies that $f\in[x\restriction m]\setminus[x\restriction(m+1)]$ and thus attests that the inclusion is proper. Then $x\in[x\restriction(m+1)]\subsetneq A$, which proves that $A$ is not an atom. We conclude that $\At(\CA({}^\omega2))=\emptyset$.
 
 Given any $A\in\CA({}^\omega2)$, since $A$ is open and because of Proposition~\ref{prop:base}, there exists $B\subseteq\mathcal B$ such that $A=\bigcup B$. Because $A$ is closed in the compact space ${}^\omega2$, $A$ is compact and, being $B$ an open cover of $A$, it has some finite subcover. Therefore, each element of $\CA({}^\omega2)$ is the finite union of elements of $\mathcal B$, which is a countable collection. Therefore  $\CA({}^\omega2)$ is countable.
 
 In summary, $(\CA({}^\omega2),\subseteq)$ is a countable atomless Boolean algebra; furthermore, according to Theorem~\ref{teo:iso_bool}, it is the only one up to isomorphism. With this we complete the section.

\section{An application to random graphs}
 In this section we will explore a use of the Rasiowa-Sikorski Lemma in a topic of great interest to Graph Theory. Let us then consider the following. We will not use more than the most basic definitions of graphs; for more details, we recommend the first two chapters of \cite{bondy}. Recall that a {\sl graph} is an ordered pair $(V,E)$ where $V$ is a set, whose elements are called {\sl vertices}, and $E$ is a set of pairs of vertices called {\sl edges}; two vertices joined by an edge are {\sl adjacent}. If the set of edges is clear, it is usual to refer to $V$ as the graph. Given a vertex $v$, $N(v)$ denotes the set of vertices that are adjacent to $v$ and their elements are called {\sl neighbors} of $v$. Two graphs are {\sl isomorphic} if there is a bijection between the sets of vertices that satisfies that two vertices are adjacent in the first graph if and only if the images of said vertices under the isomorphism are adjacent in the second graph. Let's look at two relevant examples of graphs.
 
 Suppose that $M$ is a transitive and countable model of a sufficiently large fragment of Set Theory (for an explanation of the existence of $M$, we refer the reader to sections 1 and 9 of chapter VII of \cite{kunen}). Then two sets $x,y\in V$ will be adjacent if and only if $x\in y$ or $y\in x$. Thus we form a graph with $M$. For the second example, let's take $\mathbb N$ as the set of vertices and form a graph $N$ as follows. For each pair of natural numbers, we will choose randomly and with probability $1/2$, if they are adjacent or not. Naturally, this process can produce many different graphs, but what is the probability that the two graphs $M$ and $N$ are isomorphic? The aim of this section is to answer this question, of course, using the material discussed so far.
 
  Let's start by thinking about what properties our graphs have in common.
 
\begin{definition}\label{def:ext}
 A graph $(V,E)$ has the {\sl extension property} if for any two finite disjoint sets of vertices $A$ and $B$, there exists a vertex $v\in V\setminus(A\cup B)$ such that $A\subseteq N(v)$ but $B\cap N(v)=\emptyset$.
\end{definition}

 Due to the restriction that the vertex $v$ does not belong to $A$ or $B$, it is clear that a graph with this property necessarily has an infinite number of vertices.

\begin{lemma}
 $M$ fulfills the extension property and the probability that $G$ fulfills it is $1$.
\end{lemma}

\begin{proof}
 Given $A$ and $B$ subsets of $M$ as in Definition~\ref{def:ext}, for being a model of Set Theory and satisfying, in particular, the axioms of Pairing, Union and Foundation , $v:=A\cup\{B\}$ is a vertex of $M$ with the desired property.
 
  Similarly, we take $A$ and $B$ in $G$. For a vertex $v$, the probability that it does not satisfy the desired conclusion is $1-2^{-(|A|+|B|)}$. The probability that no vertex of the infinite set $\mathbb N\setminus(A\cup B)$ satisfies the condition, by independence, is $$\lim_{k\to\infty}\left(1-2^{-(|A|+|B|)}\right)^k=0,$$ from where, taking complements, the result is obtained.
\end{proof}

 This lemma, together with Theorem~\ref{teo:random}, proves not only that our two graphs are isomorphic, but that, except for isomorphism, they are the only graph with the previously mentioned property.

\begin{theorem}\label{teo:random}
 Any two graphs over a countable set of vertices that satisfy the extension property are isomorphic.
\end{theorem}
 
 Again, the plan is to construct the isomorphism using Theorem~\ref{teo:r-s}. It is worth noting that we are constructing an isomorphism of graphs, not of order, in contrast to the results of the previous sections.
 
\begin{proof}
 Let $(V_0,E_0)$ and $(V_1,E_1)$ be two graphs with the extension property and where $V_0$ and $V_1$ are countable (and consequently infinite). Let's define $\mathbb P$ by means of the formula $p\in\mathbb P$ if and only if $p\in\Fn(V_0,V_1)$ is a bijection and for all $x,y\in\dom(p)$, $(x,y)\in E_0$ is equivalent to $(p(x),p(y))\in E_1$. Sorted by inverse inclusion, clearly $\emptyset$ is, as usual in these constructions, a maximum element of the preorder.
 
 For each $x\in V_0$ and each $y\in V_1$, we define the sets $D_x:=\{p\in\mathbb P\colon x\in\dom(p)\}$ and $D_y':=\{p\in\mathbb P\colon y\in\img(p)\}$. To see that these sets are dense in $\mathbb P$, take take $p\in\mathbb P\setminus D_x$. Let's apply the extension property that has $(V_1,E_1)$ to the finite disjoint sets $A:=p[N(x)]$ and $B:=p[V_0\setminus N(x)]$ to obtain a vertex $v\in V_1$ with the qualities that this property guarantees. Then, we assert that $q:=p\cup\{(x,v)\}$ is an element of $\mathbb P$ such that $q\supseteq p$, that is, a bijection that attests the density of $D_x$. Similarly, the extension property in $(V_0,E_0)$ can be used to prove the density of $D_y'$.
 
 Applying Theorem~\ref{teo:r-s} to the countable family of dense sets $\{D_x\colon x\in V_0\}\cup\{D_y'\colon y\in V_1\}$, we produce a generic filter $F$. Using arguments that are repetitive at this point, it can be verified that $f:=\bigcup F$ is an isomorphism between the two graphs.
\end{proof}

 Having two examples of graphs with the property of extension, we have tested the existence and uniqueness, with which we can name the graph that stars in this section as {\sl Rado graph}, also known as {\sl random graph }. It is worth mentioning some of the most remarkable properties that the Rado graph has.
 
 Using the extension property as defined, it is a simple exercise to demonstrate the following:

\begin{theorem}

 Let us denote the Rado graph by $R$.
 
\begin{enumerate}
 \item Adding or subtracting any finite number of vertices or edges from $R$ produces an isomorphic graph to $R$.
 \item Given a finite partition of the vertices of $R$, one of the classes induces (in the usual sense of Graph Theory) a graph isomorphic to $R$. Furthermore, of the non-trivial and non-complete countable graphs, $R$ is, up to isomorphism, the only one with this property (see \cite[Proposition~4,p.~5]{random}).
 \item $R$ is isomorphic to the graph obtained by inverting all adjacencies and non-adjacencies of $R$, that is, it is {\sl self-complementary}.
\end{enumerate}
\end{theorem}

 We refer the reader interested in more properties and alternative constructions of $R$, which also cover several areas of mathematics, to the article \cite{random}, where, by the way, a version of Theorem~\ref{teo:random} is proved via {\sl back-and-forth}.

\bibliographystyle{plain}
\bibliography{AnAlternativeToBackAndForth}

\end{document}